\documentclass[12pt]{article}

\usepackage{amsmath,amsthm,amsfonts,amssymb,amscd,cite}
\usepackage{ytableau}

\allowdisplaybreaks[4]

\newtheorem {lemma}{Lemma}[section]
\newtheorem {theorem} {Theorem}[section]
\newtheorem {conjecture}{Conjecture}[section]
\newtheorem {corollary}{Corollary}[section]

\usepackage{fullpage}

\usepackage{tikz}
\usepackage[margin=1in]{geometry}

\newif\ifhighlighted

\makeatletter
\newcommand{\checkHighlight}[3]{%
  \highlightedfalse
  \@for\pos:=#3\do{%
    \expandafter\checkPair\pos\relax{#1}{#2}%
  }%
}

\def\checkPair#1#2\relax#3#4{%
  \edef\r{#1}%
  \edef\c{#2}%
  \edef\checkr{#3}%
  \edef\checkc{#4}%
  \ifx\r\checkr
    \ifx\c\checkc
      \highlightedtrue
    \fi
  \fi
}
\makeatother

\newcommand{\drawYoungTableauHighlight}[4]{%
  \begin{tikzpicture}[x=#2, y=#3]
    \foreach [count=\i] \row in {#1} {
      \foreach [count=\j] \entry in \row {
        \checkHighlight{\i}{\j}{#4}%
        \pgfmathsetmacro{\x}{\j-1}
        \pgfmathsetmacro{\y}{-\i+1}
        \ifhighlighted
          \filldraw[fill=lightgray, thick] (\x,\y) rectangle ++(1,-1);
        \else
          \draw[thick] (\x,\y) rectangle ++(1,-1);
        \fi
        \node at (\x+0.5,\y-0.5) {\entry};
      }
    }
  \end{tikzpicture}%
}

\begin{document}

\title{Proof of a conjecture on eigenvalues of transposition graph}

\author{Cheng Yeaw Ku\footnote{Division of Mathematical Sciences, School of Physical and Mathematical Sciences, Nanyang Technological University, 21 Nanyang link, 
Singapore 637371, Singapore. Email: cyku@ntu.edu.sg.} \and Leyou Xu\footnote{School of Mathematical Sciences, South China Normal University, Guangzhou 510631, China. Email: leyouxu@m.scnu.edu.cn. Supported by the China Scholarship Council program (Project ID: 202406750056).} }

\date{}
\maketitle

\begin{abstract}

The transposition graph $Cay(S_n,T_n)$ is the Cayley graph on the symmetric group $S_n$ generated by the set $T_n$ of all transpositions. In this paper, we show that each integer in the interval $\left[-{\lfloor(2n+1)/3 \rfloor\choose 2}, {\lfloor(2n+1)/3 \rfloor\choose 2}\right]$ is an eigenvalue of $Cay(S_n,T_n)$. This proves a recent conjecture by Kravchuk \cite{Kravchuk}.

\medskip

\end{abstract}

\section{Introduction}

Let $S_n$ be the symmetric group on $[n]:=\{1,\dots,n\}$, and $T_n$ be the set of transpositions in $S_n$. The Cayley graph of $S_n$ generated by $T_n$, denoted by $Cay(S_n,T_n)$, is called the {\em transposition graph}  whose vertex set is $S_n$, where two vertices $f,g\in S_n$ are adjacent if and only if $fg^{-1}\in T_n$.

The eigenvalues of $Cay(S_n,T_n)$ are precsiely the eigenvalues of its corresponding adjacency matrix. 
It was established in  \cite{KY} that all eigenvalues of $Cay(S_n,T_n)$ are integers.
Since $Cay(S_n,T_n)$ is ${n\choose 2}$-regular and bipartite, the largest eigenvalue is equal to ${n\choose 2}$,  and the spectrum of $Cay(S_n,T_n)$ is symmetric about $0$; see \cite{BH}. Consequently, each eigenvalue of $Cay(S_n,T_n)$ lies within the interval $\left[-{n\choose 2},{n\choose 2}\right]$. As the spectrum of graph eigenvalues generally reveals key algebraic and structural information, it is natural to investigate the distribution of the eigenvalues of $Cay(S_n,T_n)$. Konstantinova and Kravchuk \cite{KK,KK2} demonstrated that all integers in certain intervals of linear length are eigenvalues of $Cay(S_n,T_n)$. 
More recently, Kravchuk \cite{Kravchuk} improved upon these results by proving the following theorem.

\begin{theorem}\cite{Kravchuk}\label{EK}
Let $n\ge 31$. All integers in the interval $[-n,n]$ are eigenvalues of $Cay(S_n,T_n)$.
\end{theorem}

Furthermore, Kravchuk showed in the same work that there exists an interval  of quadratic length with respect to $n$ in which each integer is an eigenvalue of $Cay(S_n,T_n)$.

\begin{theorem}\cite{Kravchuk}\label{Kra}
Let $n\ge 48$. All integers in the interval $[-y,-x]$ and $[x,y]$ are eigenvalues of $Cay(S_n,T_n)$, where $x={\lceil n/3 \rceil+1\choose  2}-2\left(\lfloor\frac{2n}{3}\rfloor-1\right)$  and $y={\lfloor (2n+1)/3\rfloor\choose 2}$
\end{theorem}

These results describe intervals of consecutive integers that appear in the spectrum of the transposition graph, while leaving a gap spanning from linear to quadratic in $n$, where the presence of eigenvalues are yet to be confirmed. Motivated by this, the following conjecture has been proposed:

\begin{conjecture}
There exists a constant $n_0$ such that for $n\ge n_0$, all integers in the interval 
\[ \left[-{\lfloor (2n+1)/3\rfloor\choose 2}, ~{\lfloor (2n+1)/3\rfloor\choose 2}\right] \]
 are eigenvalues of $Cay(S_n,T_n)$.
\end{conjecture}

In this paper, we prove that this conjecture holds.

\begin{theorem}\label{c}
For $n\ge 76$, all integers in the interval 
\[ \left[-{\lfloor(2n+1)/3 \rfloor\choose 2},~ {\lfloor(2n+1)/3 \rfloor\choose 2}\right] \]
 are eigenvalues of $Cay(S_n,T_n)$.
\end{theorem}

To this end, we establish the following result, which is a quadratic improvement (provided $n \ge 45$) of the linear interval $[-n, n]$ in Theorem \ref{EK}.
\begin{theorem}\label{x}
For $n\ge 27$, all integers in the interval $[-{\lfloor(n-15)/3 \rfloor\choose 2}, {\lfloor(n-15)/3 \rfloor\choose 2}]$ are eigenvalues of $Cay(S_n,T_n)$.
\end{theorem}

Our approach is structured as follows. After presenting preliminary results in Section 2, we proceed to prove Theorem~\ref{x} in Section~\ref{main}. By Theorem~\ref{Kra}, all integers in the interval $\left[{\lfloor n/3 \rfloor\choose 2}, {\lfloor (2n+1)/3\rfloor\choose 2}\right]$ are known to be eigenvalues of $\mathrm{Cay}(S_n, T_n)$. Since the spectrum of $\mathrm{Cay}(S_n, T_n)$ is symmetric about $0$, it remains to show that all integers in the interval $\left[{\lfloor (n - 15)/3 \rfloor\choose 2}, {\lfloor n/3 \rfloor\choose 2} \right]$ also appear as eigenvalues. This will be established in Section~4.

\section{Preliminaries}

Let $G$ be a finite group and $S\subseteq G$ be an inverse closed subset of $G$. The Cayley graph $Cay(G,S)$ is {\em normal} if $S$ is a union of conjugacy classes of $G$. It is well known that the eigenvalues of a normal Cayley graph can be calculated by the following expression.

\begin{theorem}\cite{Babai,DS,Lubotzky,Murty}\label{eigen}
The eigenvalues of a normal Cayley graph $Cay(G,S)$ are given by
\[
\rho_\chi=\sum_{g\in S}\frac{\chi(g)}{\chi(1)},
\]  
where $\chi$ ranges over all the irreducible characters of $G$ and $1$ is the identity of $G$.
\end{theorem}

A {\em partition} $\lambda$ of $n$, denoted by $\lambda \vdash n$, is a weakly decreasing sequence of positive integers $\lambda_1\ge \dots\ge \lambda_k$ with $\sum_{i=1}^k\lambda_i=n$. We write $\lambda=(\lambda_1,\dots,\lambda_k) \vdash n$. We also use the notation $(\mu_1\times t_1,\dots,\mu_r\times t_r)\vdash n$ to 
denote the partition where the parts $\mu_1 \ge \dots \ge \mu_r$  are repeated $t_1,\dots,t_r$ times respectively. 

It is well known that there is a bijection between the irreducible characters of $S_n$ and the partitions of $n$ \cite{Sagan}. 
Since $Cay(S_n,T_n)$ is normal, the eigenvalues $\rho_\lambda$ of $Cay(S_n,T_n)$ are indexed by $\lambda\vdash n$.
From Eq.~(5.6) in \cite[p.140]{Murnahan} and Theorem \ref{eigen}, for a partition $\lambda=(\lambda_1,\dots,\lambda_k)\vdash n$, its corresponding eigenvalue is equal to 
\[
\rho_\lambda=\sum_{i=1}^k\frac{\lambda_i(\lambda_i-2i+1)}{2}.
\]

Given a partition $\lambda=(\lambda_1,\dots,\lambda_k)\vdash n$, the Young diagram of $\lambda$ is a collection of $n$ boxes arranged in left-justified rows, with row lengths given by $\lambda_1$, $\ldots$, $\lambda_k$. The $(i, j)$-th box of the Young diagram is the box in the $i$-th row (from top) and the $j$-th column (from left). Let $T_\lambda$ be the generalized tableau of shape $\lambda$ by filling $t_{ij}:= j-i$ in the $(i, j)$-th box of the Young diagram. Note that the sum of $i$-th row  of $T_{\lambda}$ is then equal to \[
\sum_{j=1}^{\lambda_i}t_{ij} = \sum_{j=1}^{\lambda_i}(j-i)=\frac{\lambda_i(\lambda_i-2i+1)}{2},
\]
and so the sum of all the rows of $T_{\lambda}$ is equal to $\rho_\lambda$. This gives another way to describe the eigenvalues of $Cay(S_n,T_n)$. 

\begin{theorem}\label{eigencal}
The eigenvalues of  the transposition graph $Cay(S_n,T_n)$ are given by the sum of all the boxes of $T_\lambda$, where $\lambda$ ranges over all the partitions of $n$.
\end{theorem}

For a partition $\lambda$, the {\em $i$-th arm} of $T_{\lambda}$ is the set of all the boxes (together with the fillings) in the $i$-th row of $T_{\lambda}$ that are on the right-hand side of the $(i, i)$-th box. Similarly, the {\em $i$-th leg} corresponds to all the boxes in the $i$-th column of $T_{\lambda}$ that are below the $(i, i)$-th box. Theorem \ref{eigencal} says that the eigenvalue $\rho_\lambda$ can be obtained by summing across all the arms and legs of $T_{\lambda}$; see Figure \ref{fig:eigenvalue} for an example.

\begin{figure}[h!]
\centering
\[
T_\lambda = \begin{ytableau}
0 &  1 & 2 & 3 \\
-1 &  0 & 1 \\
 -2
\end{ytableau}
~~~\longrightarrow~~~\rho_{\lambda} =\underbrace{1+2+3}_{\textnormal{1$^{\textnormal{st}}$ arm}}+ 
\underbrace{(-1-2)}_{\textnormal{1$^{\textnormal{st}}$ leg}}+ \underbrace{1}_{\textnormal{2$^{\textnormal{nd}}$ arm}} = 4.
\]
\caption{A generalized tableau $T_{\lambda}$ and its corresponding eigenvalue.}
\label{fig:eigenvalue}
\end{figure}

\section{Proof of Theorem \ref{x}} \label{main}

Since $Cay(S_n,T_n)$ is bipartite, to prove Theorem \ref{x}, it suffices to show that each integer in $\left[0,{\lfloor(n-15)/3\rfloor\choose 2}\right]$ is an eigenvalue of $Cay(S_n,T_n)$. By \cite[Lemmas 3--4]{KK}, both $0$ and $1$ are eigenvalues of $Cay(S_n,T_n)$. Given the decomposition
\[
\left[1,{\lfloor(n-15)/3\rfloor\choose 2}\right]=\bigcup_{a=3}^{\lfloor(n-15)/3\rfloor}\left[{a-1\choose 2},{a\choose 2}\right],
\]
it suffices to show that for any given $a=3,\dots,\lfloor\frac{1}{3}(n-15)\rfloor$, every integer in $\left[{a-1\choose 2}+1,{a\choose 2}\right]$ is an eigenvalue of $Cay(S_n,T_n)$.

Notice that ${a-1 \choose 2}+1 = {a \choose 2} - (a-2)$. Our strategy is to construct a partition that corresponds to the eigenvalue
\[ {a \choose 2} - c, \]
for $c = 0, 1, \ldots, a-2$. We split the cases based on the parity of $n-a$. Theorem \ref{x} follows immediately from the following results.

\begin{theorem}[The odd case]\label{x1}
Let $n\ge 24$. For each integer $a=3,\dots,\lfloor\frac{1}{3}(n-15) \rfloor$ with $n-a\equiv 1\pmod 2$, all integers in the interval $\left[{a-1\choose 2}+1,{a\choose 2}\right]$ are eigenvalues of $Cay(S_n,T_n)$. 
\end{theorem}

\begin{theorem}[The even case]\label{x2}
Let $n\ge 27$. For each integer $a=3,\dots,\lfloor\frac{1}{3}(n-15)\rfloor$ with $n - a\equiv 0\pmod 2$, all integers in the interval $\left[{a-1\choose 2}+1,{a\choose 2}\right]$ are eigenvalues of $Cay(S_n,T_n)$. 
\end{theorem}

\subsection{The odd case}

In this subsection, assume that $n-a \equiv 1 \pmod 2$. Our goal is to construct a partition $\lambda = (\mu_1 \times t_1, \ldots, \mu_r \times t_r)  \vdash n$ that yields the desired eigenvalue. The various bounds on $n$ and $a$ ensure that the part sizes of $\lambda$ are valid. As the calculations are routine, they are omitted.

\begin{lemma}\label{odd}
Let $n\ge 7$. For each integer $a=2,\dots,\left\lfloor\frac{1}{3}(n-1)\right\rfloor$ with $n-a\equiv 1\pmod 2$, the partition \[
\left(\frac{1}{2}(n-a+1),a+1,1\times \frac{1}{2}(n-a-3)\right)
\] 
corresponds to the eigenvalue $\displaystyle {a\choose2}$ of $Cay(S_n, T_n)$.
\end{lemma}
\begin{proof}
Let $\lambda=\left(\frac{1}{2}(n-a+1),a+1,1\times \frac{1}{2}(n-a-3)\right)$ with $\frac{1}{2}(n-a-3) \ge 1$. Since the first arm and the first leg of $T_\lambda$ have the same size, these boxes cancel each other out and contribute nothing to the eigenvalue. Thus, the eigenvalue $\rho_\lambda$ is given by the sum of the boxes in the second arm of $T_{\lambda}$, i.e. 
\[
\rho_\lambda=\sum_{j=3}^{a+1}(j-2)={a\choose 2}.
\]
\end{proof}

\begin{lemma}\label{odd1}
Let $n\ge 9$. For each integer $a=2,\dots,\left\lfloor\frac{1}{3}(n-1)\right\rfloor$ with $n-a\equiv 1\pmod 2$, the partition \[
\left(\frac{1}{2}(n-a-1),a+1,3,1\times \frac{1}{2}(n-a-7)\right)
\] 
corresponds to the eigenvalue $\displaystyle {a\choose2}-1$ of $Cay(S_n, T_n)$.
\end{lemma}
\begin{proof}
Let $\lambda=\left(\frac{1}{2}(n-a-1),a+1,3,1\times \frac{1}{2}(n-a-7)\right)$ with $\frac{1}{2}(n-a-7) \ge 0$. Since the first arm and the first leg of $T_\lambda$ have the same size, these boxes cancel each other out. Thus, the eigenvalue $\rho_\lambda$ is given by the sum of the second arm and leg of $T_{\lambda}$, i.e. 
\[
\rho_\lambda=\sum_{j=3}^{a+1}(j-2)+\sum_{i=3}^3(2-j)={a\choose 2}-1.
\]
\end{proof}

\begin{lemma}\label{odd2}
Let $n\ge 17$. For each integer $a=4,\dots,\left\lfloor\frac{1}{5}(n+3)\right\rfloor$ with $n-a\equiv 1\pmod 2$ and $b=2,\dots,a-2$, the partition \[
\left(\frac{1}{2}(n-a+1-2b),a+1,b+2,2\times (b-1),1\times \frac{1}{2}(n-a-3-4b)\right)
\]
corresponds to the eigenvalue $\displaystyle {a\choose2}-b$ of $Cay(S_n, T_n)$.
\end{lemma}
\begin{proof}
Let $\lambda=\left(\frac{1}{2}(n-a+1-2b),a+1,b+2,2\times (b-1),1\times \frac{1}{2}(n-a-3-4b)\right)$ with $\frac{1}{2}(n-a-3-4b) \ge 1$.

The first arm and the first leg of $T_\lambda$ have the same size, which is $\frac{1}{2}(n-a+1-2b)$. Thus, the eigenvalue $\rho_\lambda$ is given by the sum of the second and third arms, and the second leg of $T_{\lambda}$, i.e.
\begin{align*}
\rho_\lambda&=\sum_{j=3}^{a+1}(j-2)+\sum_{j=4}^{b+2}(j-3)+\sum_{i=3}^{b+2}(2-i)\\
&={a\choose 2}-b.\qedhere
\end{align*}
\end{proof}

Combining Lemmas~\ref{odd}--\ref{odd2}, we obtain the eigenvalues in the interval $ \left[{a-1\choose 2} + 1, {a\choose 2}\right]$  for  $3 \le a \le \left\lfloor \frac{1}{5}(n + 3) \right\rfloor$.

\begin{corollary}\label{corodd}
Let $n\ge 17$. For each integer $a=3,\dots,\left\lfloor\frac{1}{5}(n+3)\right\rfloor$ with $n-a\equiv 1\pmod 2$, all integers in the interval $\left[{a-1\choose 2}+1,{a\choose 2}\right]$ are eigenvalues of $Cay(S_n,T_n)$.
\end{corollary}

Next, we turn to the remaining  cases of $a$,  where $\left\lceil\frac{1}{5}(n+3) \right\rceil\le a\le \left\lfloor\frac{1}{3}(n-15) \right\rfloor$. These bounds on $a$ imply that $n\ge 45$. The next two lemmas give the eigenvalues ${a \choose 2} - c$ over the following ranges for $c$:
\[ 2 \le c \le b+1, ~~b+2 \le c  \le a-2, \]
where
\begin{align}
b & = \frac{1}{2}(n-3a-1). \tag{*}\label{b}
\end{align}

\begin{lemma}\label{odd3}
Let $n \ge 45$.  For each integer $a = \left\lceil\frac{1}{5}(n+3)\right\rceil, \ldots,  \left\lfloor\frac{1}{3}(n-15) \right\rfloor$ with $n-a \equiv 1 \pmod{2}$, and $b$ as given in (\ref{b}), the integers ${a \choose 2} - c$, where $2 \le c \le b+1$, are eigenvalues of $Cay(S_n, T_n)$, and their associated partitions are given in Table \ref{tab:table1}.
\begin{table}[h!]
\centering
\begin{tabular}{@{}|p{0.4cm}p{6cm}|p{9cm}|@{}}
\hline
& $\rho_\lambda$ &  $\lambda$   \\
\hline 
& & \\[-0.1cm]
 & $ {a\choose 2}-2$ & $\left(a,a-1,b,7,6,4,2\times (b-7),1\times (a-b-1)\right)$ \\[0.2cm]
  & ${a\choose 2}-3$ & $\left(a\times 2,b,6\times 2,3,2\times (b-7), 1\times (a-b) \right)$ \\[0.2cm]
  & ${a\choose 2}-c, ~~~c=4, ..., \lceil b/2\rceil$ & $(a\times 2,b-c+5,2+c,4,3\times (c-3),2\times (b+2-2c), 1\times (a-b+c-5)) $\\[0.2cm]
  & ${a\choose 2}-c, ~~~c=\lceil b/2\rceil+1, ..., b-2$ & $(a\times 2,c+2,b-c+4,6,3\times (b-c-2),2\times (2c-b-1), 1\times (a-c-3))$ \\[0.2cm]
  & ${a\choose 2}-(b-1)$ & $\left(a\times 2,b,7,4,3,2\times (b-6),1\times (a-b-1)\right)$ \\[0.2cm]
  & ${a\choose 2}-b$ & $\left(a\times 2,b+1,6,4,2\times (b-4),1\times (a-b-2)\right)$ \\[0.2cm]
  & ${a\choose 2}-(b+1)$ &$\left(a,a-1,b,7,5,4,2\times (b-6),1\times (a-b-2)\right)$ \\[0.2cm]
\hline
\end{tabular}
\caption{Eigenvalues  ${a \choose 2} - c$, where $c = 2, \ldots, b+1$.}
\label{tab:table1}
\end{table}
\end{lemma}

\begin{proof}
The partition $\lambda$ is constructed so that specific arms and legs of $T_{\lambda}$ cancel each other out as much as possible, leaving a collection of boxes that add up to the desired eigenvalue. The proof is best visualized using the diagram of $T_\lambda$. For instance, the eigenvalues in the third and fouth row of Table \ref{tab:table1} are shown in Figure \ref{eigenvalue_01} and Figure \ref{eigenvalue_02} respectively. Notice that the sum of the grey boxes is to the eigenvalue. Some of these rows may not be strictly decreasing. The diagrams for all the other cases are similar and hence omitted.

\begin{figure}[h!]

\drawYoungTableauHighlight
  {{0,1,2,3, \dots, \dots, \dots, $a{-}1$},
  {$-$1,0, 1,2,\dots, \dots, \dots, $a{-}2$},
  {$-$2,$-$1, 0, 1, \dots, \dots,  $b{-}c{+}2$},
  {$-$3, $-$2,$-$1,0,\dots, $c{-}2$},
  {$-$4, $-$3, $-$2, $-$1},
  {$-$5, $-$4, $-$3},
    {\dots, \dots, \dots},
  {${-}(c{+}1)$, ${-}c$, ${-}(c{-}1)$},
  {${-}(c{+}2)$, ${-}(c{+}1)$},
  {\dots, \dots},
    {${-}(b{-}c{+}3)$, ${-}(b{-}c{+}2)$},
      {${-}(b{-}c{+}4)$},
      {\dots},
      {${-}(a{-}2)$}
  }
  {2cm}{0.5cm}
  {{1}{8}, {1}{7}, {1}{6}, {1}{5}, {1}{4}, {1}{3}, {1}{2}, {8}{3}, {5}{4}}
  
\caption{Eigenvalue ${a\choose 2}-c$, where  $c=4, ..., \lceil b/2\rceil$.}
\label{eigenvalue_01}
\end{figure}

\begin{figure}[h!]

\drawYoungTableauHighlight
  {{0,1,2,3, 4, 5,\dots, \dots, $a{-}1$},
  {$-$1,0, 1,2,3, 4, \dots, \dots, $a{-}2$},
  {$-$2,$-$1, 0, 1,2, 3,  \dots, $c{-}1$},
  {$-$3, $-$2,$-$1,0,1, 2,\dots, $b{-}c$},
  {$-$4, $-$3, $-$2, $-$1, 0, 1},
  {$-$5, $-$4, $-$3},
  {\dots, \dots, \dots},
  {${-}(b{-}c{+}2)$,${-}(b{-}c{+}1)$, ${-}(b{-}c)$},
    {\dots, \dots},
    {${-}(c{+}1)$, ${-}c$},
      {${-}(c{+}2)$},
      {\dots},
      {${-}(a{-}2)$}
  }
  {2cm}{0.5cm}
  {{1}{8}, {1}{9}, {1}{7}, {1}{6}, {1}{5}, {1}{4}, {1}{3}, {1}{2}, {10}{2}}

\caption{Eigenvalue ${a\choose 2}-c$, where $c=\lceil b/2\rceil+1, ..., b-2$.}
\label{eigenvalue_02}
\end{figure}

\end{proof}

\begin{lemma}\label{odd4}
Let $n \ge 45$.  For each integer $a = \left\lceil\frac{1}{5}(n+7) \right\rceil, \dots, \left\lfloor\frac{1}{3}(n-11) \right\rfloor$ with $n-a \equiv 1 \pmod{2}$, and $b$ as given in (\ref{b}), the integers ${a \choose 2} - c$, where $b+2 \le c \le a-2$, are eigenvalues of $Cay(S_n, T_n)$, and their associated partitions are given in Table \ref{tab:table2}.
\begin{table}[h!]
\centering
\begin{tabular}{@{}|p{0.4cm}p{6cm}|p{9cm}|@{}}
\hline
& $\rho_\lambda$ &  $\lambda$   \\
\hline 
& & \\[-0.1cm]
 & ${a\choose 2}-c, ~~~c = b+2, ... \lfloor(a+b)/2\rfloor$ & $(a,a+b-c+2,c+1,4,2\times (c-3),1\times (a+b-2c)$ \\[0.2cm]
 & & \\[-0.1cm]
  & ${a\choose 2}-c, ~~~c = \lfloor(a+b)/2\rfloor, ..., a-2$ & $(a,c+1,a+b-c-1,6\times 2,2\times (a+b-c-7),1\times (2c-a-b+3)$ \\[0.2cm]
\hline
\end{tabular}
\caption{Eigenvalues  ${a \choose 2} - c$, where $c = b+2, ...,  a-2$.}
\label{tab:table2}
\end{table}
\end{lemma}

\begin{proof}  As shown in Figure \ref{fig:eigenvalue_03} and Figure  \ref{fig:eigenvalue_04}, the sum of the grey boxes in the corresponding $T_{\lambda}$ is equal to the desired eigenvalue.

\begin{figure}[h!]
\drawYoungTableauHighlight
  {{0,1,2,3,\dots, \dots, $a{-}1$},
  {$-$1,0, 1,2, \dots, \dots, $a{+}b-c$},
  {$-$2,$-$1, 0, 1, \dots, $c{-}2$},
  {$-$3, $-$2, $-$1, 0},
  {$-$4, $-$3},
    {\dots, \dots},
    {${-}(c)$, ${-}(c{-}1)$},
      {${-}(c{+}1)$},
      {\dots},
      {${-}(a{+}b{-}c)$}
  }
  {2cm}{0.5cm}
  {{1}{8}, {1}{9}, {1}{7}, {1}{6}, {1}{5}, {1}{4}, {1}{3}, {1}{2}, {7}{2}, {4}{3}}
\caption{Eigenvalue ${a\choose 2}-c$, where  $c=b+2, ..., \lfloor (a+b)/2 \rfloor$.}
\label{fig:eigenvalue_03}
\end{figure}

\begin{figure}[h!]
\drawYoungTableauHighlight
  {{0,1,2,3,4, 5,\dots, \dots, $a{-}1$},
  {$-$1,0, 1,2, 3, 4, \dots, $c{-}1$},
  {$-$2,$-$1, 0, 1,2,3, \dots, $a{+}b{-}c{-}4$},
  {$-$3, $-$2, $-$1, 0, 1, 2},
  {$-$4, $-$3, $-$2, $-$1, 0, 1},
  {$-$5, $-$4},
    {\dots, \dots},
    {{\footnotesize${-}(a{+}b{-}c{-}3)$},{\footnotesize${-}(a{+}b{-}c{-}4)$}},
      {{\footnotesize${-}(a{+}b{-}c{-}2)$}},
      {\dots},
      {${-}c$}
  }
  {2cm}{0.5cm}
  {{1}{8}, {1}{9}, {1}{7}, {1}{6}, {1}{5}, {1}{4}, {1}{3}, {1}{2}, {11}{1}}
\caption{Eigenvalue ${a\choose 2}-c$, where $c= \lceil (a+b)/2 \rceil, \ldots, a-2$.}
\label{fig:eigenvalue_04}
\end{figure}
\end{proof}

We are now ready to prove Theorem \ref{x1}.

\begin{proof}[Proof of Theorem \ref{x1}]
Let $n \ge 24$, and $a \in \{3, \ldots, \left\lfloor \frac{1}{3}(n-15)\right\rfloor\}$ with $n-a \equiv 1 \pmod{2}$. By Lemma \ref{odd} and Lemma \ref{odd1}, we only need to consider the case ${a \choose 2}-c$, where $c = 2, \ldots, a-2$.

If $n\le 44$, then the result follows by Corollary \ref{corodd} since $\lfloor \frac{1}{5}(n+3)\rfloor\ge \lfloor\frac{1}{3}(n-15)\rfloor$. Suppose that $n\ge 45$.  By Corollary \ref{corodd}, we may assume that $\lceil \frac{1}{5}(n+3)\rceil\le a \le \lfloor\frac{1}{3}(n-15)\rfloor$. Let $b=\frac{1}{2}(n-3a-1)$ be as given in (\ref{b}). As $a\ge \frac{1}{5}(n+3)$, we have $a\ge b+2$. If $b+1\ge a-2$, then the result follows by Lemma \ref{odd3}. If $b+2\le a-2$, then $a\ge \lceil\frac{1}{5}(n+7)\rceil$ and so the result follows by Lemma \ref{odd3} and Lemma \ref{odd4}.
\end{proof}

\subsection{The even case}

In this subsection, assume that $n-a \equiv 0 \pmod 2$. As before, we will construct a partition $\lambda = (\mu_1 \times t_1, \ldots, \mu_r \times t_r)  \vdash n$ that yields the desired eigenvalue.

\begin{lemma}\label{even}
Let $n\ge 13$. For each integer $a=3,\dots,\left\lfloor\frac{1}{3}(n-4)\right\rfloor$ with $n - a \equiv 0 \pmod 2$, the partition \[
\left(\frac{1}{2}(n-a-2),a+1,4,1\times \frac{1}{2}(n-a-8)\right)
\]
corresponds to the eigenvalue ${a\choose2}$ of $Cay(S_n, T_n)$.
\end{lemma}
\begin{proof}
Let $\lambda=\left(\frac{1}{2}(n-a-2),a+1,4,1\times \frac{1}{2}(n-a-8)\right)$ with $\frac{1}{2}(n-a-8) \ge 1$. There are exactly $\frac{1}{2}(n-a-2)$ boxes in the first arm and the first leg of $T_{\lambda}$, the sum of which is $0$ and contributes nothing to the eigenvalue. Thus, the eigenvalue can be calculated using only the second arm, the second leg and the third arm of $T_{\lambda}$:
\[
\rho_\lambda=\sum_{i=3}^{a+1}(i-2)+(-1) + 1 ={a\choose 2}.
\]
\end{proof}

\begin{lemma}\label{even1}
Let $n\ge 10$. For each integer $a=2,\dots,\left\lfloor\frac{1}{3}(n-2)\right\rfloor$ with $n - a\equiv 0\pmod 2$, the partition \[
\left(\frac{1}{2}(n-a),a+1,2,1\times \frac{1}{2}(n-a-6)\right)
\]
corresponds to the eigenvalue $\displaystyle {a\choose2}-1$ of $Cay(S_n, T_n)$.
\end{lemma}
\begin{proof}
Let $\lambda=\left(\frac{1}{2}(n-a),a+1,2,1\times \frac{1}{2}(n-a-6)\right)$ with $\frac{1}{2}(n-a-6) \ge 1$. Except for the second arm and the seocnd leg of $T_{\lambda}$, all the other boxes cancel out yielding the eigenvalue
\[\rho_\lambda=\sum_{i=3}^{a+1}(i-2)-1={a\choose 2}-1. \]
\end{proof}

\begin{lemma}\label{even2}
Let $n\ge 27$. For each integer $a=5,\dots,\left\lfloor\frac{1}{3}(n-12)\right\rfloor$ with $n-a\equiv 0\pmod 2$, the partition \[
\left(\frac{1}{2}(n-a-10),a+1,6,5,3,1\times \frac{1}{2}(n-a-20)\right)
\]
corresponds to the eigenvalue $\displaystyle {a\choose2}-2$ of $Cay(S_n, T_n)$.
\end{lemma}
\begin{proof}
Let $\lambda=\left(\frac{1}{2}(n-a-10),a+1,6,5,3,1\times \frac{1}{2}(n-a-20)\right)$ with $\frac{1}{2}(n-a-20) \ge 1$. Except for the second arm and the last entry in the the third leg of $T_{\lambda}$, all the other boxes cancel out yielding the eigenvalue
\[
\rho_\lambda=\sum_{i=3}^{a+1}(i-2)+(-2)={a\choose 2}-2.
\]
\end{proof}

\begin{lemma}\label{even3}
Let $n\ge 19$. For each integer $a=5,\dots,\left\lfloor\frac{1}{3}(n-4)\right\rfloor$ with $n-a \equiv 0\pmod 2$, the partition \[
\left(\frac{1}{2}(n-a-2),a+1,3,2,1\times \frac{1}{2}(n-a-10)\right)
\]
corresponds to the eigenvalue $\displaystyle {a\choose2}-3$ of $Cay(S_n, T_n)$.
\end{lemma}
\begin{proof}
Let $\lambda=\left(\frac{1}{2}(n-a-2),a+1,3,2,1\times \frac{1}{2}(n-a-10)\right)$. The eigenvalue is just the sum of the second arm and the second leg of $T_{\lambda}$ since all the other boxes cancel out:
\[
\rho_\lambda=\sum_{i=2}^{a+1}(i-2)-1-2={a\choose 2}-3. \]
\end{proof}

\begin{lemma}\label{even4}
Let $n\ge 26$. For each integer $a=6,\dots,\lfloor\frac{1}{5}(n+4)\rfloor$ with $n-a\equiv 0 \pmod 2$ and $b=4,\dots,a-2$, the partition \[
\left(\frac{1}{2}(n-a-2b+2),a+1,b+1,3,2\times (b-3),1\times \frac{1}{2}(n-a-4b)\right)
\]
corresponds to the eigenvalue $\displaystyle {a\choose2}-b$ of $Cay(S_n, T_n)$.
\end{lemma}
\begin{proof}
As shown in Figure \ref{fig:eigenvalue_05}, the eigenvalue is just the sum of the grey boxes, which is equal to $\displaystyle {a \choose 2} - b$.

\begin{figure}[h!]
\drawYoungTableauHighlight
  {{0,1,2,3,\dots, \ldots, \dots, \footnotesize{$(n{-}a{-}2b)/2$}},
  {$-$1,0, 1,2,\ldots,  \dots, $a{-}1$},
  {$-$2,$-$1, 0, 1, \dots, $b{-}2$},
  {$-$3, $-$2, $-$1},
  {$-$4, $-$3},
    {\dots, \dots},
    {${-}b$,{${-}(b{-}1)$}},
      {\dots},
      {\scriptsize{${-}(n{-}a{-}2b)/2$}}
  }
  {2cm}{0.5cm}
  {{2}{7}, {2}{6}, {2}{5}, {2}{4}, {2}{3}, {7}{2}, {4}{3}}
\caption{Eigenvalue ${a\choose 2} - b$, where $b= 4, \ldots, a-2$.}
\label{fig:eigenvalue_05}
\end{figure}
\end{proof}

Combining the preceding lemmas, we obtain the eigenvalues in the interval $\left[{a-1 \choose 2}+1, {a \choose 2}\right]$ for $3 \le a \le \lfloor\frac{1}{5}(n+4)\rfloor$.

\begin{corollary}\label{coreven}
Let $n\ge 27$. For each integer $a=3,\dots,\lfloor\frac{1}{5}(n+4)\rfloor$ with $n- a \equiv 0\pmod 2$, all integers in the interval $[{a-1\choose 2}+1,{a\choose 2}]$ are eigenvalues of $Cay(S_n,T_n)$.
\end{corollary}

\begin{proof}
If $a =3$, then $[{a-1\choose 2}+1,{a\choose 2}] = [2, 3]$, and the result follows by Lemmas \ref{even}--\ref{even1}. If $a=4$, then   $[{a-1\choose 2}+1,{a\choose 2}] = [4, 6]$, and the result follows by  Lemmas \ref{even}--\ref{even1} and the fact that the partition $\left(\frac{1}{2}(n-a-8),6,2\times 3,\frac{1}{2}(n-a-18) \right)$ corresponds to the eigenvalue $4$. For $a \ge 5$, the result follows by Lemmas \ref{even}--\ref{even4}.
\end{proof}

It remains to consider the cases where $\left\lceil\frac{1}{5}(n+4) \rceil\le a\le \right \lfloor\frac{1}{3}(n-15)\rfloor$. These bounds implies that $n\ge 45$. In the next two lemmas, we construct the partitions that yield the eigenvalue ${a \choose 2}-c$ over the following ranges of $c$:
\[ 2 \le c \le b+1,~~ b+2 \le c \le a-2,  \]
where
\begin{align}
b & = \frac{1}{2}(n-3a+2). \tag{**}\label{bb}
\end{align}

\begin{lemma}\label{even5}
Let $n \ge 45$.  For each integer $a = \left\lceil\frac{1}{5}(n+4)\right\rceil, \ldots,  \left\lfloor\frac{1}{3}(n-15) \right\rfloor$ with $n-a \equiv 0 \pmod{2}$, and $b$ as given in (\ref{bb}), the integers  ${a \choose 2} - c$, where $2 \le c \le b+1$, are eigenvalues of $Cay(S_n, T_n)$, and their associated partitions are given in Table \ref{tab:table3}.
\begin{table}[h!]
\centering
\begin{tabular}{@{}|p{0.4cm}p{6cm}|p{9cm}|@{}}
\hline
& $\rho_\lambda$ &  $\lambda$   \\
\hline 
& & \\[-0.1cm]
  & ${a\choose 2}-c, ~c=2, ..., \lfloor (b-1)/2\rfloor$ & $(a\times 2, 3+b-c , 3+c, 3\times (c-1), 2 \times (b-2c-1), 1\times (a-b+c-3)) $\\[0.2cm]
  & ${a\choose 2}-\lceil b/2\rceil $ & $(a, a-1, (b+7)/2, (b+5)/2, 5, 3 \times (b-5)/2, 1\times (a-(b+9)/2))$ for odd $b$; \\[0.2cm]
  & & $(a, a-1, (b+8)/2, (b+6)/2, 5, 3 \times (b-2)/2, 1\times (a-b/2-5))$ for even $b$ \\[0.2cm]
  & ${a\choose 2}-c, ~c=\lceil b/2\rceil+ 1, ..., b-1$ & $(a\times 2, 2+c, b-c+4, 3 \times (b-c-1), 2 \times (2c-b-1), 1 \times (a-c-3))$ \\[0.2cm]
  & ${a\choose 2}-b$ & $(a, a-1, b-1, 6 \times 2, 3, 2 \times(b-7), 1 \times(a-b-1))$ \\[0.2cm]
  & ${a\choose 2}-(b+1)$ &$(a \times 2, b, 5, 4, 2 \times (b-5), 1 \times (a-b-1))$ \\[0.2cm]
\hline
\end{tabular}
\caption{Eigenvalues  ${a \choose 2} - c$, where $c = 2, \ldots, b+1$.}
\label{tab:table3}
\end{table}
\end{lemma}

\begin{proof}
The partition $\lambda$ is constructed so that specific arms and legs of $T_{\lambda}$ cancel each other out as much as possible, leaving a collection of boxes that add up to the desired eigenvalue. For instance, the eigenvalues in the first and third row of Table \ref{tab:table3} are shown  in Figure \ref{eigenvalue_03} and Figure \ref{eigenvalue_04} respectively, where the sum of the grey boxes is equal to the eigenvalue (all the other boxes cancel each other out). The diagrams for all the other cases are similar and hence omitted.

\begin{figure}[h!]

\drawYoungTableauHighlight
  {{0,1,2,3, \dots, \dots, \dots, $a{-}1$},
  {$-$1,0, 1,2,\dots, \dots, \dots, $a{-}2$},
  {$-$2,$-$1, 0, 1, \dots, \dots,  $b{-}c$},
  {$-$3, $-$2,$-$1,0,\dots, $c{-}1$},
  {$-$4, $-$3, $-$2},
    {\dots, \dots, \dots},
  {${-}(c{+}2)$, ${-}(c{+}1)$, ${-}c$},
  {${-}(c{+}3)$, ${-}(c{+}2)$},
  {\dots, \dots},
    {${-}(b{-}c{+}1)$, ${-}(b{-}c)$},
      {${-}(b{-}c{+}2)$},
      {\dots},
      {${-}(a{-}2)$}
  }
  {2cm}{0.5cm}
  {{1}{8}, {1}{7}, {1}{6}, {1}{5}, {1}{4}, {1}{3}, {1}{2}, {8}{3}, {5}{4}, {7}{3}}
  
\caption{Eigenvalue ${a\choose 2}-c$, where  $c=2, ..., \lceil (b-1)/2\rceil$.}
\label{eigenvalue_03}
\end{figure}

\begin{figure}[h!]

\drawYoungTableauHighlight
  {{0,1,2,3, 4, 5,\dots, \dots, $a{-}1$},
  {$-$1,0, 1,2,3, 4, \dots, \dots, $a{-}2$},
  {$-$2,$-$1, 0, 1,2, 3,  \dots, $c{-}1$},
  {$-$3, $-$2,$-$1,0,1, 2,\dots, $b{-}c$},
  {$-$4, $-$3, $-$2},
  {\dots, \dots, \dots},
  {${-}(b{-}c{+}2)$,${-}(b{-}c{+}1)$, ${-}(b{-}c)$},
    {\dots, \dots},
    {${-}(c{+}1)$, ${-}c$},
      {${-}(c{+}2)$},
      {\dots},
      {${-}(a{-}2)$}
  }
  {2cm}{0.5cm}
  {{1}{8}, {1}{9}, {1}{7}, {1}{6}, {1}{5}, {1}{4}, {1}{3}, {1}{2}, {10}{2}, {9}{2}}

\caption{Eigenvalue ${a\choose 2}-c$, where $c=\lceil b/2\rceil+1, ..., b-1$.}
\label{eigenvalue_04}
\end{figure}

\end{proof}

\begin{lemma}\label{even6}
Let $n \ge 45$.  For each integer $a = \left\lceil\frac{1}{5}(n+10)\right\rceil, \ldots,  \left\lfloor\frac{1}{3}(n-2) \right\rfloor$ with $n-a \equiv 0 \pmod{2}$, and $b$ as given in (\ref{bb}), the integers  ${a \choose 2} - c$, where $b+2 \le c \le a-2$, are eigenvalues of $Cay(S_n, T_n)$, and their associated partitions are given in Table \ref{tab:table3}.
\begin{table}[h!]
\centering
\begin{tabular}{@{}|p{0.4cm}p{6cm}|p{9cm}|@{}}
\hline
& $\rho_\lambda$ &  $\lambda$   \\
\hline 
& & \\[-0.1cm]
  & ${a\choose 2}-c, ~c=b+2, ..., \lfloor (a+b-2)/2\rfloor$ & $(a, a+b-c+1, c+1, 3, 2 \times (c-3), 1 \times (a+b-2c-1)) $\\[0.2cm]
  & ${a\choose 2}-c, ~c=\lceil (a+b-2)/2\rceil, ..., a-2$ & $(a, c+1, a+b-c, 5, 2 \times (a+b-c-5), 1 \times (2c-a-b+2))$ \\[0.2cm]
\hline
\end{tabular}
\caption{Eigenvalues  ${a \choose 2} - c$, where $c = b+2, \ldots, a-2$.}
\label{tab:table4}
\end{table}
\end{lemma}

\begin{proof}  As shown in Figure \ref{eigenvalue_05} and Figure  \ref{eigenvalue_06}, the sum of the grey boxes in the corresponding $T_{\lambda}$ is equal to the desired eigenvalue.

\begin{figure}[h!]
\drawYoungTableauHighlight
  {{0,1,2,3,\dots, \dots, $a{-}1$},
  {$-$1,0, 1,2, \dots, \dots, \footnotesize{$a{+}b-c-1$}},
  {$-$2,$-$1, 0, 1, \dots, $c{-}2$},
  {$-$3, $-$2, $-$1},
  {$-$4, $-$3},
    {\dots, \dots},
    {${-}(c)$, ${-}(c{-}1)$},
      {${-}(c{+}1)$},
      {\dots},
      {\footnotesize{${-}(a{+}b{-}c{-}1)$}}
  }
  {2cm}{0.5cm}
  {{1}{8}, {1}{9}, {1}{7}, {1}{6}, {1}{5}, {1}{4}, {1}{3}, {1}{2}, {7}{2}, {4}{3}}
\caption{Eigenvalue ${a\choose 2}$, where  $c=b+2, ..., \lfloor (a+b-2)/2 \rfloor$.}
\label{eigenvalue_05}
\end{figure}

\begin{figure}[h!]
\drawYoungTableauHighlight
  {{0,1,2,3,4, 5,\dots, \dots, $a{-}1$},
  {$-$1,0, 1,2, 3, 4, \dots, $c{-}1$},
  {$-$2,$-$1, 0, 1,2,3, \dots, $a{+}b{-}c{-}3$},
  {$-$3, $-$2, $-$1, 0, 1},
  {$-$4, $-$3},
    {\dots, \dots},
    {{\footnotesize${-}(a{+}b{-}c{-}2)$},{\footnotesize${-}(a{+}b{-}c{-}3)$}},
      {{\footnotesize${-}(a{+}b{-}c{-}1)$}},
      {\dots},
      {${-}c$}
  }
  {2cm}{0.5cm}
  {{1}{8}, {1}{9}, {1}{7}, {1}{6}, {1}{5}, {1}{4}, {1}{3}, {1}{2}, {10}{1}}
\caption{Eigenvalue ${a\choose 2}$, where $c= \lceil (a+b-2)/2 \rceil, \ldots, a-2$.}
\label{eigenvalue_06}
\end{figure}

\end{proof}


We are now ready to prove Theorem \ref{x2}.

\begin{proof}[Proof of Theorem \ref{x2}]
If $\lfloor \frac{1}{5}(n+4) \rfloor\ge \lfloor\frac{1}{3}(n-15)\rfloor$, then the result follows by Corollary \ref{coreven}. Suppose in the following that $\lceil \frac{1}{5}(n+4)\rceil \le \lfloor\frac{1}{3}(n-15) \rfloor$. Then $n\ge 45$. By Corollary \ref{coreven}, we may assume that $\lceil \frac{1}{5}(n+4) \rceil \le a \le  \lfloor\frac{1}{3}(n-15)\rfloor $. Let $b=\frac{1}{2}(n-3a+2)$.  Then $a > b$. If $b+1\ge a-2$, then the result follows by Lemma \ref{even5}. If $b+1 <  a-2$, then $a\ge \frac{1}{5}(n+10)$ and so the result follows by  Lemmas \ref{even5} and \ref{even6}.
\end{proof}

\subsection{Small eigenvalues}

The following facts can be verified easily by direct calculations.
\begin{itemize}
\item[(a)] The following partitions correspond to the eigenvalue $2$ of  $Cay(S_n,T_n)$:
\begin{center}
\begin{tabular}{cc}
$\displaystyle \left(\frac{n-3}{2},4,2,1\times \frac{n-9}{2} \right)$ & for odd $n \ge 11$;\\
$\displaystyle \left(\frac{n-4}{2},4,3,1\times  \frac{n-10}{2} \right)$ & for even $n \ge 12$.
\end{tabular}
\end{center}
\item[(b)] The following partitions correspond to the eigenvalue $3$ of  $Cay(S_n,T_n)$:
\begin{center}
\begin{tabular}{cc}
$\displaystyle \left(\frac{n-5}{2},4\times 2,1\times\frac{n-11}{2} \right)$ & for odd $n \ge 13$;\\
$\displaystyle \left(\frac{n-2}{2},4,1\times\frac{n-6}{2} \right) $ & for even $n \ge 10$.
\end{tabular}
\end{center}
\end{itemize}

It is worth noting that the bounds for $n$ above are best possible. Indeed, using Sagemath, it is found that $2$ is not an eigenvalue of $Cay(S_n,T_n)$ for $n=9,10$ and $3$ is not an eigenvalue of $Cay(S_n,T_n)$ for $n=8,11$.

Note that for $15\le n\le 26$, we have $\left\lfloor\frac{n-15}{3} \right\rfloor\le 3$. It follows from the facts above and Theorem \ref{x} that the following result holds.

\begin{theorem}\label{xx}
For $n\ge 15$, all integers in the interval $\left[-{\lfloor(n-15)/3 \rfloor\choose 2},{\lfloor(n-15)/3\rfloor\choose 2}\right]$ are eigenvalues of $Cay(S_n,T_n)$.
\end{theorem}

\section{Proof of Theorem \ref{c}}

Armed with the preceding results, we are now ready to prove the main result. We will also require the following lemma proved in \cite{Kravchuk}.

\begin{lemma}\label{l}\cite{Kravchuk}
Let $\lambda=\left(\lambda_1,\dots,\lambda_k\right)$ be a partition of $n$. Then, with $\lambda'=\left(\lambda_2,\dots,\lambda_k \right)$
\[
\rho_\lambda=\rho_{\lambda'}+{\lambda_1\choose 2}-(n-\lambda_1).
\]
\end{lemma}

\begin{proof}[Proof of Theorem \ref{c}]
By Theorem \ref{Kra} and the fact that the spectrum of $\mathrm{Cay}(S_n, T_n)$ is symmetric about $0$, it suffices to show that all integers in $[{\lfloor (n-15)/3 \rfloor\choose 2},{\lfloor n/3 \rfloor\choose 2} ]$ are eigenvalues of $Cay(S_n,T_n)$. 

Choose a partition $\lambda = (\lambda_1, \ldots, \lambda_k) \vdash n$, and let $\lambda' = (\lambda_2, \ldots, \lambda_k) \vdash n'$, where $n'=n- \lambda_1$. Applying Theorem \ref{x} to $Cay(S_{n'}, T_{n'})$, all integers in $\left[-{\lfloor (n'-15)/3 \rfloor \choose 2},  {\lfloor (n'-15)/3\rfloor \choose 2}\right]$ are eigenvalues of $Cay(S_{n'}, T_{n'})$. Moreover, the first rows of the partitions associated with these eigenvalues (based on the constructions in Section 3) must be at least $\lfloor\frac{1}{2}(n-\lambda_1-2)\rfloor$. Thus, for $\lambda$ to be a valid partition, it is necessary that $\lambda_1 \ge  \lfloor\frac{1}{2}(n-\lambda_1-2)\rfloor$, i.e. $\lambda_1 \ge \frac{1}{3}(n-2)$.

Choose  $\lambda_1=\lfloor\frac{n}{3}\rfloor$. By Lemma \ref{l}, all integers in
\[ \left[-{\lfloor (n'-15)/3 \rfloor \choose 2}+{\lambda_1\choose 2}-n',  {\lfloor (n'-15)/3\rfloor \choose 2}+{\lambda_1\choose 2}-n'\right] \]
are eigenvalues of $Cay(S_n, T_n)$.

Since $n\ge 76$, it is readily verified that
\begin{eqnarray*}
-{\lfloor (n'-15)/3 \rfloor \choose 2}+{\lambda_1\choose 2}-n' & \le & {\lfloor (n-15)/3\rfloor\choose 2 },  \\
{\lfloor (n'-15)/3\rfloor \choose 2}+{\lambda_1\choose 2}-n' & \ge & {\lfloor n/3\rfloor\choose 2 }.
\end{eqnarray*}
This implies that each integer in the interval $[{\lfloor(n-15)/3\rfloor\choose 2},{\lceil n/3\rceil\choose 2}]$ is an eigenvalue of $Cay(S_n,T_n)$, completing the proof of the theorem.
\end{proof}


\end{document}